\documentclass[11pt]{article}
\usepackage[T1]{fontenc}
\usepackage[margin=1in]{geometry}
\usepackage{amsmath,amssymb,amsthm,mathtools}
\usepackage{enumitem}
\usepackage{microtype}
\usepackage{xcolor}
\usepackage[
  colorlinks=true,
  linkcolor=blue!45!black,
  citecolor=blue!45!black,
  urlcolor=blue!55!black
]{hyperref}

\newtheorem{theorem}{Theorem}
\newtheorem{lemma}[theorem]{Lemma}

\newtheorem{remark}[theorem]{Remark}

\newtheorem{question}[theorem]{Question}

\newcommand{\R}{\mathbb R}
\newcommand{\la}{\langle}
\newcommand{\ra}{\rangle}
\newcommand{\norm}[1]{\left\lVert #1\right\rVert}
\newcommand{\abs}[1]{\left\lvert #1\right\rvert}
\newcommand{\dd}{\,\mathrm d}

\title{Local Well-Posedness for Vlasov--Poisson\\
with $L^{d+}$ Initial Density and Fractional Velocity Regularity}
\author{
  {\bf Quoc-Hung Nguyen\thanks{
      E-mail address: qhnguyen@amss.ac.cn.
      State Key Laboratory of Mathematical Sciences,
      Academy of Mathematics \& Systems Science,
      The Chinese Academy of Sciences, Beijing 100190, China,
      and School of Mathematical Sciences,
      University of Chinese Academy of Sciences,
      Beijing 100049, China.}}
}
\date{}
\hypersetup{
  pdftitle={Local Well-Posedness for Vlasov--Poisson with Ld-plus Initial Density and Fractional Velocity Regularity},
  pdfauthor={Quoc-Hung Nguyen},
  pdfsubject={Local well-posedness for the Vlasov--Poisson system},
  pdfkeywords={Vlasov--Poisson, fractional velocity regularity, mixing estimate, Schauder fixed point}
}
\begin{document}
\maketitle

\begin{abstract}
We prove a local well-posedness criterion for the Vlasov--Poisson system on
$\R^d_x\times\R^d_v$, $d\geq2$, under an anisotropic assumption on the initial
distribution.  The datum has finite mass, its weighted velocity supremum
belongs to $L^p_x$ for some $p>d$, and it has an arbitrarily small positive
H\"older regularity in the velocity variable, uniformly with respect to
velocity and with the same spatial $L^p$ control.  The main estimate is a
nonlinear mixing bound
\[
 [\rho(t)]_{C^\alpha_x}\lesssim
 t^{-d/p-\epsilon}C(f_0),
 \qquad \epsilon>0~~\text{small}.
\]
Thus the density is integrable in time with values in a positive spatial
H\"older class, and the corresponding electric field belongs to
$L^1_tC^{1,\alpha}_x$.  We construct a solution by a
Schauder fixed point on the density and prove uniqueness by a
Loeper-type stability estimate.
\end{abstract}



\section{Introduction}

\subsection{Background and motivation}
We consider, in dimension $d\geq 2$, the Vlasov--Poisson system
\begin{equation}\label{VP}
 \partial_t f+v\cdot\nabla_x f+E\cdot\nabla_v f=0,
 \qquad E=\sigma\nabla_x\Delta_x^{-1}\rho,
 \qquad \rho(t,x)=\int_{\R^d}f(t,x,v)\dd v,
\end{equation}
where $\sigma\in\{-1,1\}$.  The sign plays no role in the local theory.
The Vlasov--Poisson equation is a basic kinetic model for collisionless
plasmas and self-gravitating matter.  Its transport structure suggests solving
the equation along the characteristics generated by the phase-space vector
field $(v,E)$.  At low regularity, however, this description is useful only
after one has shown that the spatial density $\rho$ is regular enough to make
the characteristic flow well defined and stable.  This is the central
difficulty: $f$ lives on phase space, while the field is generated by the
velocity average $\rho=\int f\,\dd v$.

\subsection{Related literature and contribution}

General introductions to kinetic equations and the Vlasov--Poisson system can
be found in \cite{BGP,Glassey,Rein07}.  The classical Cauchy theory was
developed by Horst \cite{Horst,Horst2}.  Weak solutions and moment propagation
are treated in \cite{Ars,CPallard}, while the transport and Lagrangian
structures for rough solutions are developed in
\cite{DiPernaLions,Lagrangian}.  Local theories in weighted Sobolev,
Bessel-potential, and Besov spaces appear in \cite{WFS1,WFS2,Besov}.

Global existence should be distinguished from global well-posedness
in the present low-regularity class.  In dimension $d=2$, global classical
solutions were
established by Ukai and Okabe \cite{UkaiOkabe}; in dimension $d=3$,
corresponding results were obtained by Pfaffelmoser, Lions--Perthame, and
Schaeffer \cite{Pfaffelmoser,LionsPerthame,Schaeffer}.  These results require
substantially stronger regularity, decay, or velocity-moment assumptions.
For nonsmooth vector fields, the theory of maximal regular
flows was developed by Ambrosio, Colombo, and Figalli
\cite{MaximalFlows}; its application to the Vlasov--Poisson system gives a
global Lagrangian existence theory for rough finite-energy data in the
repulsive case in dimensions $d\geq3$ \cite{Lagrangian}.  In particular, in
the three-dimensional repulsive case,
if finite kinetic energy were added to \eqref{data-assumption}, then the
initial field energy would also be finite: indeed,
$\rho_0\in L^1_x\cap L^p_x$ with $p>3$ implies
$\rho_0\in L^{6/5}_x$ and hence
$\nabla\Delta^{-1}\rho_0\in L^2_x$.  The resulting finite-energy theory
provides global Lagrangian existence, but it does not propagate the
$L^1_{\rm loc}C^1_x$ field regularity or yield global uniqueness in the
anisotropic class considered here.  In dimension $d=2$, a nonzero total
charge produces a Coulomb field which is not in $L^2_x$ at infinity, so an
analogous energy argument requires a neutralized or renormalized
formulation \cite[Remark~2.10]{Lagrangian}.  We therefore restrict the
present paper to local
well-posedness; global well-posedness under the assumptions of
Theorem~\ref{thm:main} remains outside the scope of the present argument.

The weighted Sobolev, Bessel-potential, and Besov theories
cited above impose regularity on the datum in the full phase space.\
Jeong and Tae
\cite{JT} were able to pass below that regime  in dimensions $d\geq3$.  For compactly
velocity-supported data they proved local well-posedness in
\begin{equation*}
 H^s(\R^d_x\times\R^d_v),\qquad s>\frac d2-\frac{1}{4},
\end{equation*}
so the distribution need not be bounded.  Tae \cite{Tae} subsequently
developed an $H^{s,p}$ theory and constructed the solution through a
density fixed-point argument.

Those arguments  use  Sobolev regularity of  $f_0$ in the full phase space,
 together with the velocity-averaging  estimate of
 Golse--Lions--Perthame--Sentis  \cite{GLPS}.  Our assumptions are anisotropic:
they impose a weighted H\"older modulus only in the velocity variable, and
the size of this modulus is controlled in  $L^p_x$ .  In particular, no
positive-order  spatial derivative of $f_0$ is  required.  Hence the preceding
Sobolev theories do not cover the present data class.  This distinction  is
the main motivation for the present work.  We also use a density fixed point,
with compactness supplied by the
Aubin--Lions--Simon theorem \cite{Simon}, but replace the Sobolev averaging
estimate by a direct physical-space estimate based on characteristics,
ballistic mixing, and a weighted fractional modulus in velocity.

For uniqueness, the relevant reference point is Loeper's stability theorem
under a bounded spatial density \cite{Loeper}.  Extensions to certain
unbounded densities, Orlicz classes, and localized Yudovich-type conditions
are given in \cite{Uniqueness,HoldingMiot,Yudovich}.
These results do not establish uniqueness under the sole
assumption that the density belongs to $L^p_x$ for one fixed
$p<\infty$.  Instead, they require control of finite $L^q_x$ norms for
arbitrarily large $q$, with prescribed growth as $q\to\infty$, or
corresponding Orlicz and localized Yudovich-type conditions.\
Kinetic Wasserstein
distances adapted to the position--velocity geometry and refinements of
Loeper's stability estimates were developed in
\cite{Iacobelli,IacobelliJunne}.
The refinement in \cite{IacobelliJunne} covers kinetic
Wasserstein orders $1<p<\infty$; this exponent indexes the transport distance
and should not be confused with control of the density in one fixed
$L^p_x$ space.\  In the present setting the density need not
be uniformly bounded up to $t=0$; instead,  $\norm{\rho(t)}_{L^\infty}$ has an
\textit{integrable time singularity}.  This permits a pointwise-in-time use of
Loeper's estimate.

Our main result treats an anisotropic class in which phase-space Sobolev
regularity is replaced by a weighted, uniform $C^\theta_v$ modulus whose
spatial size belongs to $L^p_x$.  For every $p>d$ and every
 $0<\theta\leq1$, we
obtain local existence, uniqueness, and continuous dependence.  Polynomial
velocity decay is sufficient, and no full velocity derivative is assumed.
The mechanism can already be seen for free transport.  In that case
\begin{equation}\label{free-density-intro}
 \rho^{\mathrm{free}}(t,x)
 =\int_{\R^d}f_0(x-tv,v)\dd v
 =t^{-d}\int_{\R^d}f_0\left(y,\frac{x-y}{t}\right)\dd y.
\end{equation}
 Thus a  $C^\theta_v$ modulus of $f_0$
produces a $C^\alpha_x$ modulus of $\rho(t)$ for every
 $0<\alpha<\theta$.  The nonlinear proof shows that this ballistic mixing
survives the perturbation by a sufficiently small, time-integrated electric
field.

\subsection{Main result}

For $0<\theta\leq1$ and $m>0$,  set
$\la z\ra=(1+\abs z^2)^{1/2}$ and define  the weighted velocity envelopes
 \begin{equation}\label{G-definition}
 G_0(x)=\operatorname*{ess\,sup}_{v\in\R^d}
 \la v\ra^m\abs{f_0(x,v)},
 \qquad
 G_\theta(x)=
 \operatorname*{ess\,sup}_{\substack{v\in\R^d\\0<\abs h\leq1}}
 \la v\ra^m
 \frac{\abs{f_0(x,v+h)-f_0(x,v)}}{\abs h^\theta}.
\end{equation} 
Thus $G_\theta$ is a weighted, uniform $C^\theta$ seminorm in velocity,
whose dependence on $x$ is measured in $L^p_x$.  This is an anisotropic
H\"older assumption, not a phase-space Sobolev assumption.
For almost every $x$, the bounds in \eqref{G-definition} determine a unique
$C^\theta_{\rm loc}$ representative of the section $v\mapsto f_0(x,v)$.
We use this representative throughout (and set it equal to zero on the
exceptional set of $x$'s).  In particular, expressions such as
$f_0(y,\mathcal W_t(x,y))$ below are unambiguous.

We use the following solution concept.  A \emph{Lagrangian solution} on
$[0,T]$ is a pair $(f,E)$ for which  $E\in L^1(0,T;C^1)$, the backward
characteristic system
\begin{equation}\label{characteristics-intro}
 \partial_sX_{s,t}=V_{s,t},\qquad
 \partial_sV_{s,t}=E(s,X_{s,t}),\qquad
 (X_{t,t},V_{t,t})=(x,v)
\end{equation}
has a unique flow, and
\begin{equation*}
 f(t,x,v)=f_0\bigl(X_{0,t}(x,v),V_{0,t}(x,v)\bigr),
 \qquad E=\sigma\nabla\Delta^{-1}\int f\,\dd v.
\end{equation*}
Such a solution also solves
\eqref{VP} in the sense of distributions.

Throughout the paper, $[\,\cdot\,]_{C^\alpha}$
denotes the homogeneous H\"older seminorm.  The letter $C$ denotes a positive
constant that may change from line to line and depends only on the fixed
parameters of the theorem.

\begin{theorem}[Local well-posedness with fractional velocity regularity]\label{thm:main}
Let $d\geq2$, let $p>d$, let $m>d+1$, fix $0<\theta\leq1$, and let
$f_0\geq0$ satisfy
\begin{equation}\label{data-assumption}
 f_0\in L^1(\R^{2d}),\qquad
 G_0,G_\theta\in L^p(\R^d),
\end{equation}
where $G_0$ and $G_\theta$ are defined in \eqref{G-definition}.
Choose
\begin{equation}\label{alpha-r-choice}
 0<\alpha<\min\left\{\theta,1-\frac dp\right\},
 \qquad
 1<r<\frac{1}{d/p+\alpha}.
\end{equation}
Then there exists $T\in(0,1]$, depending only on
$d,p,m,\theta,\alpha,r$ and
\begin{equation*}
 A_0:=\norm{f_0}_{L^1_{x,v}}+\norm{G_0}_{L^p_x}+\norm{G_\theta}_{L^p_x},
\end{equation*}
such that \eqref{VP} has a unique Lagrangian solution on $[0,T]$.  It satisfies
 \begin{align}
 &f\in C\big([0,T];L^1(\R^{2d})\cap L^p(\R^{2d})\big),\label{fclass}\\
 &\rho\in L^\infty\big(0,T;L^1(\R^d)\cap L^p(\R^d)\big)
           \cap L^r\big(0,T;C^\alpha(\R^d)\big),\label{rhoclass}\\
 &E\in L^r\big(0,T;C^{1,\alpha}(\R^d)\big).\label{Eclass}
\end{align} 
Define the current density by
$j(t,x):=\int_{\R^d}v f(t,x,v)\dd v$.\
For $0<t\leq T$,
\begin{align}
 \norm{\rho(t)}_{L^\infty_x}
 +\norm{j(t)}_{L^\infty_x}
 &\leq C t^{-d/p}\norm{G_0}_{L^p_x},\label{rhoinfty}\\
 [\rho(t)]_{C^\alpha_x}
 &\leq C t^{-d/p-\alpha}
       \big(\norm{G_0}_{L^p_x}+\norm{G_\theta}_{L^p_x}\big),\label{rhoholder}\\
 \norm{\rho(t)}_{L^p_x}+\norm{j(t)}_{L^p_x}
 &\leq C\norm{G_0}_{L^p_x}.\label{currentbound}
\end{align}
In particular, $\rho\in L^1(0,T;C^\alpha_b)$ and
$E\in L^1(0,T;C^{1,\alpha}_b)$.
Moreover, on sets of initial data for which $A_0$ is uniformly bounded, the
solution map is continuous from $L^1_{x,v}\cap L^p_{x,v}$ into
$C([0,T];L^1_{x,v}\cap L^p_{x,v})$.
\end{theorem}

\begin{remark}[Interpretation of the assumptions]\label{rem:assumptions}
The condition $G_0\in L^p_x$ already implies $f_0\in L^p_{x,v}$, since
\begin{equation*}
 \norm{f_0}_{L^p_{x,v}}^p
 \leq \norm{G_0}_{L^p_x}^p
      \int_{\R^d}\la v\ra^{-mp}\dd v<\infty.
\end{equation*}
The stronger restriction $m>d+1$ is used to control the current
$j=\int vf\,\dd v$. 
\end{remark}

\subsection{Main ideas of the proof}

We briefly explain the argument and the origin of the restrictions in
Theorem~\ref{thm:main}.  The proof is based on the characteristic method.  In
particular, in backward variables we study the map from terminal velocity to
initial position, invert this map for short times, and express the density
through the resulting characteristic coordinates.  This physical-space
viewpoint, as well as the use of sharp pointwise characteristic estimates,
follows ideas developed by Huang, the author, and Xu\ for the screened Vlasov--Poisson system and for the
two-dimensional Vlasov--Poisson system with massless electrons
\cite{HNXscreened,HNX2d}.  The perturbative problems treated there are
different from the present local theory, but the characteristic
reparametrization is closely related.

\paragraph{1. A prescribed density determines a characteristic flow.}
Starting from a density $\varrho$, set
$E_\varrho=\sigma\nabla\Delta^{-1}\varrho$ and let
$(X_{s,t}^\varrho,V_{s,t}^\varrho)$ be its backward characteristics.
Transporting the fixed datum $f_0$ by this flow
produces a new density
\begin{equation*}
 \mathcal G[\varrho](t,x)
 =\int_{\R^d}f_0\bigl(X_{0,t}^\varrho(x,v),
                         V_{0,t}^\varrho(x,v)\bigr)\dd v.
\end{equation*}
A fixed point $\mathcal G[\rho]=\rho$ is exactly a solution of
Vlasov--Poisson.

\paragraph{2. The nonlinear mixing formula.}
For a prescribed field with small $L^1_tC^{1,\alpha}_x$ norm, the map from terminal
velocity to initial position is invertible.  We denote by
$\mathcal V_t(x,y)$ the terminal velocity of the characteristic whose
initial position is $y$, and by $\mathcal W_t(x,y)$ its initial velocity.
Then
\begin{equation}\label{density-formula-intro}
 \rho_E(t,x)=t^{-d}\int_{\R^d}
 f_0\bigl(y,\mathcal W_t(x,y)\bigr)J_t(x,y)\dd y,
 \qquad \mathcal W_t(x,y)=\frac{x-y}{t}+O(1).
\end{equation}
This is the nonlinear counterpart of \eqref{free-density-intro}.

\paragraph{3. Fractional velocity regularity becomes spatial regularity.}
The increment $x\mapsto x+\ell$ changes $\mathcal W_t$ by
$O(\abs\ell/t)$.
The $G_\theta$ assumption and the H\"older continuity of the Jacobian
$J_t$ yield
\begin{equation}\label{key-estimate-intro}
 [\rho_E(t)]_{C^\alpha_x}
 \leq C t^{-d/p-\alpha}
 \bigl(\norm{G_0}_{L^p_x}+\norm{G_\theta}_{L^p_x}\bigr).
\end{equation}
This is the only step that uses velocity regularity.  Since
$d/p+\alpha<1$, the right-hand side is integrable in time.  Elliptic
H\"older estimates then give
$\nabla\Delta^{-1}\rho_E\in
L^1_tC^{1,\alpha}_x$, which closes the characteristic construction.

\paragraph{4. Compact fixed point and uniqueness.}
The map $\mathcal G$ is not shown to be a contraction, so the Banach fixed
point theorem is not used.  Instead, the density bounds, the continuity
equation, and spatial tightness produce a closed convex
set of admissible densities in a Banach space, while the density map has
relatively compact range.  The Schauder fixed-point theorem then yields a
fixed point, as in \cite[Section~4.1]{NguyenMemoir}; see also the density
fixed-point construction in \cite{Tae}.\ 
Finally, the bounds
$\norm{\rho(t)}_{L^\infty}\lesssim t^{-d/p}$ and
$E\in L^1_tC^1_x$, together with\ Loeper's field stability
estimate, give uniqueness by Gronwall's inequality.

\subsection{Organization of the paper}

Section~\ref{sec:elliptic-flow} develops the short-time
characteristic reparametrization.  In particular, we prove that the map from
terminal velocity to initial position is a global diffeomorphism of $\R^d$
for each sufficiently short time, construct its inverse, and estimate the
associated Jacobian.  Section~\ref{sec:mixing}
uses this change of variables to derive exact representation formulas for
the density and current and to establish the required $L^p$, $L^\infty$,
and H\"older bounds.  In Section~\ref{sec:fixed-point}, these estimates,
together with compactness and continuity, are used to construct a solution
through a Schauder fixed point on the density.
Section~\ref{sec:uniqueness} combines a Loeper-type stability estimate with
the integrable bound on $\norm{\rho(t)}_{L^\infty}$ to prove uniqueness.
Section~\ref{sec:continuity} establishes time continuity of the solution and
continuous dependence on the initial datum.  Finally,
Section~\ref{sec:endpoints} compares the result with the classical
bounded-density theory, discusses the endpoints $p=d$ and $\theta=0$, and
formulates a related nonuniqueness problem.\\\\
{\bf Acknowledgements.} The research of Quoc-Hung Nguyen was supported by the 
CAS Project for Young Scientists in Basic Research, Grant No.~YSBR-031; and the NSFC under Grant Nos.~1251101538 and 12595282. 
\section{Characteristic estimates}\label{sec:elliptic-flow}

We first establish the short-time characteristic reparametrization that
turns velocity integration into a dispersive spatial integral.

\subsection{Characteristic reparametrization}

The purpose of this subsection is to construct the inverse
characteristic reparametrization and to quantify its deviation from the
free-transport change of variables.

Fix $0<\alpha<1$ and a field
\begin{equation}\label{field-assumption}
 E\in L^1(0,T;C^{1,\alpha}(\R^d)).
\end{equation}
Following the notation of \cite{HNXscreened,HNX2d}, let
$(X_{s,t}(x,v),V_{s,t}(x,v))$ be the backward characteristics:
\begin{equation}\label{backwardflow}
 \begin{cases}
  \dfrac{\dd}{\dd s}X_{s,t}(x,v)=V_{s,t}(x,v),
  &X_{t,t}(x,v)=x,\\[2mm]
  \dfrac{\dd}{\dd s}V_{s,t}(x,v)=E(s,X_{s,t}(x,v)),
  &V_{t,t}(x,v)=v.
 \end{cases}
\end{equation}
Thus
\begin{align}
 X_{s,t}(x,v)
 &=x-(t-s)v+\int_s^t(\tau-s)
     E(\tau,X_{\tau,t}(x,v))\dd\tau,\label{backward-X}\\
 V_{s,t}(x,v)
 &=v-\int_s^tE(\tau,X_{\tau,t}(x,v))\dd\tau.\label{backward-V}
\end{align}
The phase-space map $(x,v)\mapsto(X_{s,t},V_{s,t})$ preserves Lebesgue
measure.

\begin{lemma}[Characteristic reparametrization]\label{lem:twist}
Let $0<T\leq1$.  There exists $\delta_0=\delta_0(d)>0$ such that, if
\begin{equation}\label{small-field}
 \int_0^T\norm{E(s)}_{C^{1,\alpha}}\dd s\leq\delta_0,
\end{equation}
then for every $0\leq s<t\leq T$, there is a global $C^1$ map
$\Psi_{s,t}:\R^d_x\times\R^d_v\to\R^d_v$ satisfying
\begin{equation}\label{Psi-definition}
 X_{s,t}\bigl(x,\Psi_{s,t}(x,v)\bigr)=x-(t-s)v.
\end{equation}
For $t>0$ and $x,y\in\R^d$, set
\begin{equation}\label{VWJ-definition}
 \mathcal V_t(x,y)
 :=\Psi_{0,t}\left(x,\frac{x-y}{t}\right),\qquad
 \mathcal W_t(x,y)
 :=V_{0,t}\bigl(x,\mathcal V_t(x,y)\bigr),
\end{equation}
and
\begin{equation}\label{J-definition}
 J_t(x,y)
 :=\left|\det D_v\Psi_{0,t}
       \left(x,\frac{x-y}{t}\right)\right|.
\end{equation}
Then
\begin{equation}\label{endpoint-identities}
 X_{0,t}\bigl(x,\mathcal V_t(x,y)\bigr)=y,\qquad
 V_{0,t}\bigl(x,\mathcal V_t(x,y)\bigr)=\mathcal W_t(x,y),
\end{equation}
the map $v\mapsto X_{0,t}(x,v)$ is a global $C^1$ diffeomorphism, and
 \begin{align}
 \left|\mathcal V_t(x,y)-\frac{x-y}{t}\right|
 +\left|\mathcal W_t(x,y)-\frac{x-y}{t}\right|
 &\leq C\int_0^t\norm{E(\tau)}_{L^\infty}\dd\tau,\label{VWfree}\\
 C^{-1}\leq J_t(x,y)&\leq C,\label{twistbounds}\\
 \left|\mathcal W_t(x,y)-\mathcal W_t(x',y)\right|
 &\leq C\frac{|x-x'|}{t},\label{W-Lipschitz}\\
 |J_t(x,y)-J_t(x',y)|
 &\leq C|x-x'|^\alpha
       \int_0^t\norm{E(\tau)}_{C^{1,\alpha}}\dd\tau.\label{Jholder}
\end{align} 
All constants are uniform under \eqref{small-field}.
\end{lemma}

\begin{proof}
Put
\begin{equation*}
 A_{s,t}(x,v):=D_vX_{s,t}(x,v).
\end{equation*}
Differentiating \eqref{backward-X} gives
\begin{equation}\label{A-equation}
 A_{s,t}=-(t-s)I+
 \int_s^t(\tau-s)\nabla E(\tau,X_{\tau,t})A_{\tau,t}\dd\tau.
\end{equation}
If
 $M=\sup_{0\leq s<t\leq T}(t-s)^{-1}\norm{A_{s,t}}_{L^\infty}$, then
\begin{equation*}
 M\leq1+M\sup_{s<t}\int_s^t
 \frac{(\tau-s)(t-\tau)}{t-s}\norm{\nabla E(\tau)}_{L^\infty}\dd\tau.
\end{equation*} 
The fraction in the integrand is bounded by $T$, so
\eqref{small-field} gives $M\leq2$ after reducing $\delta_0$.  Returning to
\eqref{A-equation},
 \begin{equation}\label{A-close}
 \norm{-\frac{A_{s,t}}{t-s}-I}_{L^\infty}
 \leq C T\int_s^t\norm{\nabla E(\tau)}_{L^\infty}\dd\tau
 \leq\frac{1}{2}.
\end{equation} 
Consequently $v\mapsto X_{s,t}(x,v)$ is locally invertible.  It is proper,
because \eqref{backward-X} gives
 \begin{equation*}
 X_{s,t}(x,v)=x-(t-s)v+
 O\left((t-s)\int_s^t\norm{E(\tau)}_{L^\infty}\dd\tau\right).
\end{equation*} 
The global inverse theorem therefore shows that it is a global
diffeomorphism.  Defining $\Psi_{s,t}$ by \eqref{Psi-definition} is now
legitimate.  Differentiating that identity in $v$ yields
\begin{equation}\label{Psi-Jacobian}
 A_{s,t}\bigl(x,\Psi_{s,t}(x,v)\bigr)
 D_v\Psi_{s,t}(x,v)=-(t-s)I.
\end{equation}
In particular, \eqref{A-close} and \eqref{Psi-Jacobian} prove
\eqref{twistbounds}.

Let $u=(x-y)/t$.  Equations \eqref{backward-X},
\eqref{Psi-definition}, and \eqref{backward-V} give
\begin{align*}
 \mathcal V_t(x,y)-u
 &=\frac{1}{t}\int_0^t\tau E(\tau,\Gamma_\tau)\dd\tau,\\
 \mathcal W_t(x,y)-\mathcal V_t(x,y)
 &=-\int_0^tE(\tau,\Gamma_\tau)\dd\tau,
\end{align*}
where
\begin{equation*}
 \Gamma_\tau
 :=X_{\tau,t}\bigl(x,\mathcal V_t(x,y)\bigr).
\end{equation*}
This proves \eqref{VWfree}.

We next compare two trajectories with the same initial point $y$ and terminal
positions $x,x'$.  Let $\Gamma,\Gamma'$ be the corresponding curves and put
$D_\tau=\Gamma_\tau-\Gamma'_\tau$ and $\ell=x-x'$.
Indeed, by \eqref{Psi-definition} and
\eqref{VWJ-definition},
\begin{align*}
 \Gamma_0
 &=X_{0,t}\bigl(x,\mathcal V_t(x,y)\bigr)
   =x-t\frac{x-y}{t}=y,\\
 \Gamma'_0
 &=X_{0,t}\bigl(x',\mathcal V_t(x',y)\bigr)
   =x'-t\frac{x'-y}{t}=y,
\end{align*}
whereas the terminal conditions for the backward characteristics give
$\Gamma_t=x$ and $\Gamma'_t=x'$.  Therefore

\begin{equation*}
 D_0=0,\qquad D_t=\ell,\qquad
 \ddot D_\tau=E(\tau,\Gamma_\tau)-E(\tau,\Gamma'_\tau).
\end{equation*}
Let $G_t$ be the Dirichlet Green kernel for $-\partial_s^2$ on $[0,t]$,
namely
\begin{equation}\label{Green-kernel}
 G_t(s,\tau)=
 \begin{cases}
  \dfrac{\tau(t-s)}{t},&0\leq\tau\leq s,\\[2mm]
  \dfrac{s(t-\tau)}{t},&s\leq\tau\leq t.
 \end{cases}
\end{equation}
Hence the exact identity is
\begin{equation}\label{boundary-Green}
 D_s=\frac{s}{t}\ell-
 \int_0^tG_t(s,\tau)
 \bigl(E(\tau,\Gamma_\tau)-E(\tau,\Gamma'_\tau)\bigr)\dd\tau,
\end{equation}
and the sign will not matter in the estimates.  For fixed $t$, define
\begin{equation*}
 N_t:=\sup_{0<s\leq t}\frac{t}{s}|D_s|.
\end{equation*}
 Since $D\in C^1([0,t])$ and $D_0=0$, the quantity $N_t$ is finite.\
Then $|D_\tau|\leq(\tau/t)N_t$.  Moreover, the explicit formula
\eqref{Green-kernel} gives, for $0<s\leq t$ and $0\leq\tau\leq t$,
\begin{equation*}
 \frac{t}{s}G_t(s,\tau)\frac{\tau}{t}\leq\tau.
\end{equation*}
Therefore \eqref{boundary-Green} and the Lipschitz bound for $E$ imply
 \begin{align*}
 \frac{t}{s}|D_s|
 &\leq |\ell|+\int_0^t\frac{t}{s}G_t(s,\tau)
        \norm{\nabla E(\tau)}_{L^\infty} |D_\tau|\dd\tau\\
 &\leq |\ell|+N_t\int_0^t\tau
        \norm{\nabla E(\tau)}_{L^\infty}\dd\tau
 \leq |\ell|+T\delta_0N_t.
\end{align*} 
 Since $T\leq1$, after decreasing $\delta_0$ so that
$\delta_0\leq1/2$, we obtain\
$N_t\leq2|\ell|$.  This proves
\begin{equation}\label{two-endpoint-position}
 |D_s|\leq C\frac{s}{t}|\ell|,\qquad0\leq s\leq t.
\end{equation}
The estimate is also valid at $s=0$ because $D_0=0$.

For $0<\tau\leq t$, the second line of \eqref{Green-kernel} gives
\begin{equation*}
 \partial_sG_t(0,\tau)=\frac{t-\tau}{t},
 \qquad |\partial_sG_t(0,\tau)|\leq1.
\end{equation*}
Hence, differentiating \eqref{boundary-Green} at $s=0$, we obtain
the exact identity
\begin{equation*}
 \dot D_0=\frac{\ell}{t}
 -\int_0^t\frac{t-\tau}{t}
 \bigl(E(\tau,\Gamma_\tau)-E(\tau,\Gamma'_\tau)\bigr)\dd\tau.
\end{equation*}
Using the Lipschitz bound for $E$ and the estimate
$|D_\tau|\leq2(\tau/t)|\ell|$ obtained above, it follows that
\begin{align*}
 |\dot D_0|
 &\leq\frac{|\ell|}{t}
 +2\frac{|\ell|}{t}
   \int_0^t\tau\norm{\nabla E(\tau)}_{L^\infty}\dd\tau\\
 &\leq(1+2T\delta_0)\frac{|\ell|}{t}
 \leq2\frac{|\ell|}{t},
\end{align*}
where the last inequality follows from $T\delta_0\leq1/2$.
\
Because $\dot\Gamma_0=\mathcal W_t(x,y)$, this proves
\eqref{W-Lipschitz}.

It remains to prove \eqref{Jholder}.  Evaluate $A_{s,t}$ along $\Gamma$ and
$\Gamma'$ and denote the two matrices by $A^x_{s,t}$ and $A^{x'}_{s,t}$.
More explicitly,
\begin{equation*}
 A^x_{s,t}=D_vX_{s,t}\bigl(x,\mathcal V_t(x,y)\bigr),\qquad
 A^{x'}_{s,t}=D_vX_{s,t}\bigl(x',\mathcal V_t(x',y)\bigr).
\end{equation*}
Writing $\Delta A_{s,t}=A^x_{s,t}-A^{x'}_{s,t}$ and subtracting
\eqref{A-equation}, we obtain
\begin{align*}
 \Delta A_{s,t}
 ={}&\int_s^t(\tau-s)\nabla E(\tau,\Gamma_\tau)
                    \Delta A_{\tau,t}\dd\tau\\
 &+\int_s^t(\tau-s)
 \bigl[\nabla E(\tau,\Gamma_\tau)
       -\nabla E(\tau,\Gamma'_\tau)\bigr]
 A^{x'}_{\tau,t}\dd\tau.
\end{align*}
Set
\begin{equation*}
 H:=\sup_{0\leq s<t}\frac{\|\Delta A_{s,t}\|}{t-s}.
\end{equation*}
Using \eqref{two-endpoint-position} and
$\|A^{x'}_{\tau,t}\|\leq C(t-\tau)$, division by $t-s$ gives
 \begin{equation*}
 H\leq CT\!\left(\int_0^t\norm{\nabla E(\tau)}_{L^\infty}\dd\tau\right)H
 +C|\ell|^\alpha\int_0^t
     [\nabla E(\tau)]_{C^\alpha}\dd\tau.
\end{equation*} 
The first coefficient is at most $1/2$ after decreasing $\delta_0$.
Absorption therefore yields
\begin{equation}\label{A-holder}
 \sup_{0\leq s<t}\frac{\|A^x_{s,t}-A^{x'}_{s,t}\|}{t-s}
 \leq C|\ell|^\alpha
       \int_0^t\norm{E(\tau)}_{C^{1,\alpha}}\dd\tau.
\end{equation}
Finally, \eqref{Psi-Jacobian} at $s=0$ gives
\begin{equation}\label{J-via-A}
 J_t(x,y)=
 \left|\det\left(-t^{-1}A^x_{0,t}\right)\right|^{-1}.
\end{equation}
The determinant and its reciprocal are Lipschitz on the fixed neighborhood
of $I$ specified by \eqref{A-close}.  Equations \eqref{A-holder} and
\eqref{J-via-A} prove \eqref{Jholder}.
\end{proof}

\section{Density and current produced}\label{sec:mixing}

We now use this characteristic reparametrization to derive
an exact formula for the transported\ 
density and then obtain the dispersive bounds needed by the fixed point.

\subsection{Exact representation formulas}

We begin by changing variables from terminal velocity to initial position
along the backward characteristics.

Let $E$ satisfy the hypotheses of Lemma~\ref{lem:twist}, and let
\begin{equation}\label{transported-f}
 f_E(t,x,v)=f_0\bigl(X_{0,t}(x,v),V_{0,t}(x,v)\bigr),
 \qquad
 \rho_E(t,x)=\int f_E(t,x,v)\dd v.
\end{equation}
For fixed $(t,x)$, the map $v\mapsto y=X_{0,t}(x,v)$ is a diffeomorphism.
By \eqref{Psi-definition}, its inverse and its Jacobian are
\begin{equation*}
 v=\mathcal V_t(x,y)
   =\Psi_{0,t}\left(x,\frac{x-y}{t}\right),
 \qquad
 \dd v=t^{-d}J_t(x,y)\dd y.
\end{equation*}
Together with \eqref{endpoint-identities}, this gives the exact representation
\begin{equation}\label{densityrepresentation}
 \rho_E(t,x)=t^{-d}\int_{\R^d}
 f_0\bigl(y,\mathcal W_t(x,y)\bigr)J_t(x,y)\dd y.
\end{equation}
Similarly,
\begin{equation}\label{currentrepresentation}
 j_E(t,x)=t^{-d}\int_{\R^d}
 \mathcal V_t(x,y)f_0\bigl(y,\mathcal W_t(x,y)\bigr)
 J_t(x,y)\dd y.
\end{equation}

The factor $t^{-d}$ is the dispersive scaling, $\mathcal W_t(x,y)$ is the
nonlinear replacement of $(x-y)/t$ inside the initial datum, and $J_t$
measures the deviation of the characteristic change of variables from free
transport.  The next lemma quantifies these observations.

\subsection{The dispersive estimate}

The following estimate shows that the free-transport decay and the conversion
of velocity regularity into spatial H\"older regularity persist for a small
nonlinear field.

\begin{lemma}[Dispersive estimates]\label{lem:mixing}
Assume that $0<\alpha<\theta\leq1$.\
Under the hypotheses of Lemma~\ref{lem:twist}, for $0<t\leq\min\{T,1\}$,
 \begin{align}
 \norm{\rho_E(t)}_{L^\infty}+\norm{j_E(t)}_{L^\infty}
 &\leq Ct^{-d/p}\norm{G_0}_{L^p},\label{mix-Linfty}\\
 \norm{\rho_E(t)}_{L^p}+\norm{j_E(t)}_{L^p}
 &\leq C\norm{G_0}_{L^p},\label{mix-Lp}\\
 [\rho_E(t)]_{C^\alpha}
 &\leq Ct^{-d/p-\alpha}
       \big(\norm{G_0}_{L^p}+\norm{G_\theta}_{L^p}\big).
       \label{mix-Calpha}
\end{align} 
The constants are uniform for fields satisfying \eqref{small-field}.
\end{lemma}

\begin{proof}
By \eqref{VWfree}, after decreasing $\delta_0$ if necessary,
\begin{equation}\label{weightcomparison}
 \la \mathcal W_t(x,y)\ra^{-m}
 \leq C\la (x-y)/t\ra^{-m}.
\end{equation}
Define
\begin{equation*}
 K_t(z):=t^{-d}\la z/t\ra^{-m}.
\end{equation*}
Writing $p'=p/(p-1)$ for the H\"older conjugate of $p$, we have
 $\norm{K_t}_{L^1}\leq C$ and
$\norm{K_t}_{L^{p'}}\leq Ct^{-d/p}$.  From
\eqref{densityrepresentation}, \eqref{twistbounds}, and
\eqref{weightcomparison},
\begin{equation}\label{rho-convolution}
 \abs{\rho_E(t,x)}\leq C(K_t*G_0)(x).
\end{equation}
Young's inequality proves \eqref{mix-Linfty} and the density part of
\eqref{mix-Lp}.

For the current, \eqref{VWfree} gives
\begin{equation*}
 \abs{\mathcal V_t(x,y)}\leq C\big(1+\abs{x-y}/t\big).
\end{equation*}
The kernel
\begin{equation*}
 \widetilde K_t(z):=t^{-d}\la z/t\ra^{-m+1}
\end{equation*}
satisfies
\begin{equation*}
 \norm{\widetilde K_t}_{L^1}\leq C,\qquad
 \norm{\widetilde K_t}_{L^{p'}}\leq Ct^{-d/p}.
\end{equation*}
Indeed, the second norm is finite because $(m-1)p'>d$, which follows from
$m>d+1$.  Applying Young's inequality to \eqref{currentrepresentation}
proves the estimates for $j_E$ in both \eqref{mix-Linfty} and
\eqref{mix-Lp}.

To prove \eqref{mix-Calpha}, let $\ell\in\R^d$ and split the increment of
\eqref{densityrepresentation} as
\begin{align}
 \rho_E(t,x+\ell)-\rho_E(t,x)
 &=I_1(t,x,\ell)+I_2(t,x,\ell),\label{rho-increment-split}\\
 I_1
 &=t^{-d}\int
 \bigl[f_0(y,\mathcal W_t(x+\ell,y))
       -f_0(y,\mathcal W_t(x,y))\bigr]
 J_t(x+\ell,y)\dd y,\nonumber\\
 I_2
 &=t^{-d}\int f_0(y,\mathcal W_t(x,y))
 \bigl[J_t(x+\ell,y)-J_t(x,y)\bigr]\dd y.\nonumber
\end{align}
Suppose first that $\abs \ell\leq c t$, where $c>0$ is sufficiently small.
Then
\eqref{W-Lipschitz} implies
\begin{equation*}
 \abs{\mathcal W_t(x+\ell,y)-\mathcal W_t(x,y)}
 \leq C\frac{\abs \ell}{t}\leq1.
\end{equation*}
The definition of $G_\theta$, together with the comparability of the two
velocity weights, therefore gives
\begin{equation*}
 \abs{f_0(y,\mathcal W_t(x+\ell,y))
       -f_0(y,\mathcal W_t(x,y))}
 \leq C\left(\frac{\abs \ell}{t}\right)^\theta
 G_\theta(y)\la(x-y)/t\ra^{-m}.
\end{equation*}
Using the preceding pointwise estimate exactly as in
\eqref{rho-convolution}, Young's inequality gives
 \begin{equation}\label{I1-estimate}
 \norm{I_1(t,\cdot,\ell)}_{L^\infty}
 \leq C\abs\ell^\alpha t^{-\alpha-d/p}\norm{G_\theta}_{L^p}.
\end{equation} 
For $I_2$, the Jacobian estimate \eqref{Jholder}, the smallness condition
\eqref{small-field}, and the same convolution kernel give
 \begin{equation}\label{I2-estimate}
 \norm{I_2(t,\cdot,\ell)}_{L^\infty}
 \leq C\abs\ell^\alpha t^{-d/p}\norm{G_0}_{L^p}.
\end{equation} 
Together, \eqref{I1-estimate}--\eqref{I2-estimate} yield
 \begin{equation}\label{difference-small-h}
 \norm{\rho_E(t,\cdot+\ell)-\rho_E(t)}_{L^\infty}
 \leq C\abs \ell^\alpha t^{-\alpha-d/p}
       \big(\norm{G_0}_{L^p}+\norm{G_\theta}_{L^p}\big)
\end{equation} 
for $\abs \ell\leq ct$.  We also have, for all $\ell$,
 \begin{equation}\label{difference-zero}
 \norm{\rho_E(t,\cdot+\ell)-\rho_E(t)}_{L^\infty}
 \leq Ct^{-d/p}\norm{G_0}_{L^p}.
\end{equation} 
If $\abs \ell\leq ct$, divide \eqref{difference-small-h} by
$\abs \ell^\alpha$. If $\abs \ell>ct$, divide \eqref{difference-zero} by
$\abs \ell^\alpha$ and use $\abs \ell^{-\alpha}\leq Ct^{-\alpha}$.
This proves \eqref{mix-Calpha}.
\end{proof}

The transported density and current satisfy
\begin{equation}\label{continuity-equation}
 \partial_t\rho_E+\nabla_x\cdot j_E=0
\end{equation}
in distributions.  Hence Lemma~\ref{lem:mixing} gives
 \begin{equation}\label{time-derivative-density}
 \norm{\partial_t\rho_E}_{L^\infty(0,T;W^{-1,p})}
 \leq C\norm{G_0}_{L^p}.
\end{equation} 
Thus the time compactness needed below comes from the continuity equation and
the current estimate.  In the Sobolev theories \cite{JT,Tae}, the analogous
compactness is obtained from a velocity-averaging estimate.

We shall also use tightness.  Since
\begin{equation*}
 \abs{X_t(y,w)}
 \leq\abs y+T\abs w+T\norm E_{L^1_tL^\infty_x},
\end{equation*}
volume preservation implies
\begin{equation}\label{tail-estimate}
 \sup_{0\leq t\leq T}\int_{\abs x>R}\rho_E(t,x)\dd x
 \leq
 \iint_{\abs y+T\abs w>R-C}f_0(y,w)\dd y\dd w
 \xrightarrow[R\to\infty]{}0,
\end{equation}
uniformly for fields in a bounded set satisfying
\eqref{small-field}.

\section{The density fixed point}\label{sec:fixed-point}

The existence argument uses the compact-map form of the
classical Schauder fixed-point theorem, as in
\cite[Section~4.1]{NguyenMemoir}; see also
\cite[Section~11.1]{GilbargTrudinger}.

\begin{theorem}[Schauder fixed-point theorem]\label{thm:Schauder}
Let $\mathcal X$ be a Banach space, let
$\mathcal K\subset\mathcal X$ be nonempty, closed, and convex, and let
$F:\mathcal K\to\mathcal K$ be continuous.  If $F(\mathcal K)$ is relatively
compact in $\mathcal X$, then $F$ has a fixed point in $\mathcal K$.
\end{theorem}

We now construct $\mathcal X$, $\mathcal K$, and $F$ for
Vlasov--Poisson.

\subsection{The closed convex set of admissible
densities}

We choose a closed convex set that records the spatial
regularity, time compactness, mass control, and uniform tightness furnished
by the dispersive estimate.  The same bounds will imply that the image of
the density map is relatively compact.

Put
\begin{equation}\label{qdef}
 q=\frac dp+\alpha.
\end{equation}
By \eqref{alpha-r-choice}, $rq<1$.  Fix
$0<\beta<\alpha$ and use the Banach space
\begin{equation}\label{fixed-point-space}
 \mathcal X=L^r\big(0,T;C^\beta(\R^d)\big),\qquad
 \norm{u}_{\mathcal X}
 :=\left(\int_0^T
   \big(\norm{u(t)}_{L^\infty}+[u(t)]_{C^\beta}\big)^r\dd t\right)^{1/r}.
\end{equation}
Let  $M=\norm{f_0}_{L^1}$  and
$\rho_0(x)=\int f_0(x,v)\dd v$.

Choose constants $R_p,R_t>0$, with sufficiently large dimensional factors,
comparable to
 $\norm{G_0}_{L^p}+\norm{G_\theta}_{L^p}$, and choose
 \begin{equation}\label{Ralpha}
 R_\alpha=2C\big(\norm{G_0}_{L^p}+\norm{G_\theta}_{L^p}\big)T^{1/r-q},
\end{equation} 
where $C$ is the constant in Lemma~\ref{lem:mixing}.  Let $\mathcal K$ be
the set of nonnegative densities $\varrho$ such that
 \begin{align}
 &\norm{\varrho}_{L^\infty_t(L^1_x\cap L^p_x)}\leq M+R_p,
 \qquad \norm{\varrho}_{L^r_tC^\alpha_x}\leq R_\alpha,
 \label{K-spatial}\\
 &\norm{\partial_t\varrho}_{L^\infty_tW^{-1,p}_x}\leq R_t,
 \qquad \varrho(0)=\rho_0,\qquad
 \int_{\R^d}\varrho(t,x)\dd x=M\quad\text{for a.e. }t,
 \label{K-time}
\end{align} 
and which obey the common tightness modulus on the right-hand side of
\eqref{tail-estimate}.  We replace this set by its closure
in $\mathcal X$.\  Enlarging the harmless constants makes it
nonempty.  Indeed, the free-transport density in
\eqref{free-density-intro} satisfies the same convolution, increment,
continuity-equation, and tightness estimates with $J_t\equiv1$ and
$\mathcal V_t(x,y)=\mathcal W_t(x,y)=(x-y)/t$.  The set $\mathcal K$ is
convex because every defining
condition is convex.

\subsection{The density self-map}

We next associate with each $\varrho\in\mathcal K$ the
density obtained by transporting the fixed initial datum under the field
generated by $\varrho$.  This defines a nonlinear self-map.

For $\varrho\in\mathcal K$, define
\begin{equation}\label{fixedpointmap}
 E_\varrho=\sigma\nabla\Delta^{-1}\varrho,
 \qquad
 \mathcal G[\varrho](t,x)
 =\int_{\R^d} f_0\bigl(X_{0,t}^\varrho(x,v),
                           V_{0,t}^\varrho(x,v)\bigr)\dd v,
\end{equation}
where $(X_{s,t}^\varrho,V_{s,t}^\varrho)$ are the backward characteristics
generated by $E_\varrho$.  The electric field associated with the output
density is
\begin{equation}\label{output-field}
 \widetilde E_\varrho
 :=\sigma\nabla\Delta^{-1}\mathcal G[\varrho].
\end{equation}

\begin{lemma}[Closedness and compactness]
\label{lem:compactness}
The set $\mathcal K$ is a nonempty closed convex subset of the Banach space
$\mathcal X$, and all the bounds
\eqref{K-spatial}--\eqref{K-time} and the tightness condition remain valid
on $\mathcal K$.  Moreover, every family satisfying these bounds and the
same tightness modulus is relatively compact in $\mathcal X$.
\end{lemma}

\begin{proof}
Fix a ball $B_R$.  We have the compact and continuous embeddings
\begin{equation*}
 C^\alpha(B_R)\Subset C^\beta(B_R)
 \hookrightarrow W^{-1,p}(B_R).
\end{equation*}
The Aubin--Lions--Simon compactness theorem \cite{Simon}, using
\eqref{K-spatial} and \eqref{K-time}, gives relative compactness in
$L^r(0,T;C^\beta(B_R))$.  A diagonal argument over $R=1,2,\ldots$ gives
convergence in $L^r(0,T;C^\beta_{\rm loc}(\R^d))$ along a subsequence.

We upgrade this local convergence to convergence in the Banach space
$\mathcal X$.  Let $\omega(R)\to0$ be the common tightness modulus, so that
\begin{equation*}
 \int_{\abs x>R}\varrho(t,x)\dd x\leq\omega(R)
 \qquad\text{for a.e. }t
\end{equation*}
for every density in the family.  Positivity and the spatial H\"older bound
give, with $H_\varrho(t)=\norm{\varrho(t)}_{C^\alpha}$,
\begin{equation}\label{tail-Linfty-holder}
 \norm{\varrho(t)}_{L^\infty(\{\abs x>R+1\})}
 \leq C\omega(R)^{\alpha/(d+\alpha)}
       H_\varrho(t)^{d/(d+\alpha)}+C\omega(R).
\end{equation}
Indeed, one may average on a ball of radius $0<s\leq1$ and minimize the
bound $Cs^{-d}\omega(R)+s^\alpha H_\varrho(t)$.  Interpolation between
$C^0$ and $C^\alpha$, applied on the exterior region with a slightly
enlarged cutoff, then yields some $\kappa>0$ such that
\begin{equation}\label{tail-Cbeta}
 \norm{\varrho}_{L^r(0,T;C^\beta(\{\abs x>R+2\}))}
 \leq C\omega(R)^\kappa
       \bigl(1+\norm{\varrho}_{L^r_tC^\alpha_x}\bigr).
\end{equation}
The right-hand side tends to zero uniformly over the family.  Combining
\eqref{tail-Cbeta} with the local convergence proves relative compactness in
$\mathcal X$.

It remains to verify that the defining conditions survive passage to the
closure.  Let $\varrho_n$ belong to the set before closure and suppose that
$\varrho_n\to\varrho$ in $\mathcal X$.  After passing to a subsequence,
$\varrho_n(t)\to\varrho(t)$ in $C^\beta(\R^d)$ for almost every $t$.
Positivity and the $C^\alpha$ bound pass to the limit pointwise in time, and
Fatou's lemma gives
\begin{equation*}
 \norm{\varrho}_{L^r_tC^\alpha_x}
 \leq\liminf_{n\to\infty}
       \norm{\varrho_n}_{L^r_tC^\alpha_x}.
\end{equation*}
The uniform $L^\infty_tL^p_x$ bound passes by weak-star compactness and
identification of the local limit.  The common tightness modulus, together
with local convergence, gives for almost every $t$
\begin{equation*}
 \varrho_n(t)\longrightarrow\varrho(t)\quad\hbox{in }L^1(\R^d).
\end{equation*}
Consequently $\int\varrho(t)=M$, the $L^\infty_tL^1_x$ bound is retained,
and the same tightness modulus holds for $\varrho$.

The derivatives $\partial_t\varrho_n$ are weak-star precompact in
$L^\infty(0,T;W^{-1,p})$.  Testing against compactly supported
space--time functions identifies every weak-star limit with
$\partial_t\varrho$, so the derivative bound is retained.  Moreover, on
each ball the functions $\varrho_n$ are equi-Lipschitz in time with values
in $W^{-1,p}(B_R)$.  Their convergence in
$L^r(0,T;W^{-1,p}(B_R))$ therefore improves to convergence in
$C([0,T];W^{-1,p}(B_R))$.  Since $\varrho_n(0)=\rho_0$, the initial trace
of $\varrho$ is also $\rho_0$.  Thus the closure retains all the defining
properties.
\end{proof}

\begin{lemma}[The density self-map]\label{lem:selfmap}
For $T>0$ sufficiently small, the map $\mathcal G$ in
\eqref{fixedpointmap} satisfies
\begin{equation*}
 \mathcal G[\mathcal K]\subset\mathcal K
\end{equation*}
and is continuous from $\mathcal K$ to $\mathcal K$ in
the norm of $\mathcal X$.  Moreover, $\mathcal G[\mathcal K]$ is relatively
compact in $\mathcal X$.
\end{lemma}

\begin{proof}
\emph{Self-mapping property.}
For $\varrho\in\mathcal K$, the global Schauder estimate for the Poisson
equation gives
\begin{equation}\label{E-varrho-bound}
 \norm{E_\varrho}_{L^r_tC^{1,\alpha}_x}
 \leq C\big((M+R_p)T^{1/r}+R_\alpha\big).
\end{equation}
Consequently,
\begin{equation*}
 \norm{E_\varrho}_{L^1_tC^{1,\alpha}_x}
 \leq CT^{1-1/r}\big((M+R_p)T^{1/r}+R_\alpha\big)
 \leq C A_0T+C A_0T^{1-q}.
\end{equation*}
Because $q<1$, this is smaller than $\delta_0$ when $T$ is small.  Thus
the twist lemma and all dispersive estimates apply uniformly.

Combining \eqref{mix-Linfty} and
\eqref{mix-Calpha}, using $t^{-d/p}\leq t^{-q}$ for
$0<t\leq T\leq1$, and recalling that $rq<1$, we obtain
\begin{align}
 \norm{\mathcal G[\varrho]}_{L^r_tC^\alpha_x}
 &\leq
 C\big(\norm{G_0}_{L^p}+\norm{G_\theta}_{L^p}\big)
 \left(\int_0^Tt^{-rq}\dd t\right)^{1/r}\nonumber\\
 &\leq
 C\big(\norm{G_0}_{L^p}+\norm{G_\theta}_{L^p}\big)T^{1/r-q}
 \leq R_\alpha.\label{selfmap-holder}
\end{align}
Moreover, positivity of $f_0$, the identity of the flow at $t=0$,
and volume preservation give
\begin{equation*}
 \mathcal G[\varrho]\geq0,\qquad
 \mathcal G[\varrho](0)=\rho_0,\qquad
 \int_{\R^d}\mathcal G[\varrho](t,x)\dd x=M.
\end{equation*}
Together with \eqref{mix-Lp},
\eqref{time-derivative-density}, and \eqref{tail-estimate}, these identities
verify all the other defining conditions of $\mathcal K$.  Hence
$\mathcal G[\mathcal K]\subset\mathcal K$.
The same uniform bounds and tightness modulus, together
with Lemma~\ref{lem:compactness}, show that
$\mathcal G[\mathcal K]$ is relatively compact in $\mathcal X$.

\emph{Continuity.}
Suppose $\varrho_n\to\varrho$ in
$\mathcal X$.   We first record the convergence of the corresponding fields.
Let $O\Subset\R^d$, let $0<\beta'<\beta$, and choose a cutoff $\chi_R$
which equals one on a large ball containing $O$.  For the local part, the
Schauder estimate and the convergence in $\mathcal X$ give
\begin{equation*}
 \nabla\Delta^{-1}\bigl(\chi_R(\varrho_n-\varrho)\bigr)
 \longrightarrow0
 \quad\hbox{in }L^1(0,T;C^{1,\beta'}(O)).
\end{equation*}
For the complementary part, the Newtonian kernel and all its derivatives are
smooth when $x\in O$ and $y$ lies outside the larger ball.  Once $R$ is large
enough relative to $O$, its $C^{1,\beta'}(O)$ norm is bounded uniformly in
$R$ by a constant depending only on $O$, times
\begin{equation*}
 \int_{\abs y>R}\bigl(\varrho_n(t,y)+\varrho(t,y)\bigr)\dd y.
\end{equation*}
The common tightness modulus makes this bound uniformly small as
$R\to\infty$.  Thus, first taking $n\to\infty$ and then $R\to\infty$, we
obtain
 \begin{equation}\label{field-convergence}
 E_{\varrho_n}\longrightarrow E_\varrho
 \quad\hbox{in }L^1(0,T;C^{1,\beta'}(O)).
\end{equation}
The flows therefore converge uniformly on compact subsets of phase space,
uniformly for $t\in[0,T]$.  Indeed, all relevant trajectories remain in a
slightly larger compact set, and Gronwall applied to the difference of the
two characteristic systems uses \eqref{field-convergence} and the uniform
$L^1_tC^1_x$ bound.

If $g\in C_c^\infty(\R^{2d})$, this flow convergence and volume preservation
give
\begin{equation*}
 g\bigl(X_{0,t}^{\varrho_n},V_{0,t}^{\varrho_n}\bigr)
 \longrightarrow
 g\bigl(X_{0,t}^{\varrho},V_{0,t}^{\varrho}\bigr)
 \quad\hbox{in }C\big([0,T];L^1\cap L^p\big).
\end{equation*}
Approximate $f_0$ in $L^1\cap L^p$ by such a $g$.  Since composition with
each flow preserves every $L^a$ norm, $1\leq a<\infty$, the same conclusion
holds with $g$ replaced by $f_0$.  In particular, the transported phase-space
densities converge in distributions.  On the other hand, the family
$\{\mathcal G[\varrho_n]\}$ has the uniform estimates of
Lemma~\ref{lem:mixing} and \eqref{time-derivative-density}; by
Lemma~\ref{lem:compactness}, it is precompact in $\mathcal X$.  Every
convergent subsequence has, by the distributional convergence, the unique
limit $\mathcal G[\varrho]$.  Thus the entire sequence converges to that
limit in $\mathcal X$, proving continuity.
\end{proof}

\begin{proof}[Proof of existence in Theorem~\ref{thm:main}]
Theorem~\ref{thm:Schauder}, applied to
Lemmas~\ref{lem:compactness} and \ref{lem:selfmap}, gives a\ 
$\rho\in\mathcal K$ such that $\mathcal G[\rho]=\rho$.  Let
$E=\sigma\nabla\Delta^{-1}\rho$ and
\begin{equation*}
 f(t,x,v)=f_0\bigl(X_{0,t}(x,v),V_{0,t}(x,v)\bigr).
\end{equation*}
The flow is volume preserving, so $f$ solves the transport equation in
distributions and preserves its $L^1$ and $L^p$ norms.  The fixed-point
identity gives $\rho=\int f\dd v$, hence $(f,E)$ solves \eqref{VP}.
The estimates \eqref{rhoinfty}--\eqref{currentbound} follow from
Lemma~\ref{lem:mixing}, while \eqref{Eclass} follows from the global
Schauder estimate for the Poisson equation.
\end{proof}

\section{Uniqueness}\label{sec:uniqueness}

We use the following consequence of Loeper's field estimate
\cite{Loeper}.

\begin{lemma}[Loeper stability estimate]\label{lem:loeper}
Let $\rho_1,\rho_2\geq0$ be bounded densities with equal finite mass, let
$E_i=\nabla\Delta^{-1}\rho_i$, and let $\pi$ be a coupling of $\rho_1$ and
$\rho_2$ with finite quadratic transportation cost.  Then, with
 $R=\max_{i=1,2}\norm{\rho_i}_{L^\infty}$,
\begin{align}
 \norm{E_1-E_2}_{L^2_x}^2
 &\leq C_d R\int\abs{x-y}^2\dd\pi(x,y),\label{loeper-L2}\\
 \int\abs{E_1(x)-E_2(x)}^2\rho_1(x)\dd x
 &\leq C_d R^2\int\abs{x-y}^2\dd\pi(x,y).
 \label{loeper}
\end{align}
Only this pointwise-in-time estimate is needed; the density norms need not be
uniformly bounded in time.   In particular, no finite second moment of either
density is assumed separately.
\end{lemma}

\begin{proof}
We first remove a minor moment issue  from the usual Wasserstein formulation .
Let
\begin{equation*}
 \pi_n=\boldsymbol 1_{B_n\times B_n}\pi
\end{equation*}
and let $\rho_{1,n},\rho_{2,n}$ be its two marginals.  They have equal finite
mass, compact support, and satisfy $0\leq\rho_{i,n}\leq\rho_i\leq R$.
 Loeper's $H^{-1}$ estimate , applied after normalization  by their common mass ,
gives
\begin{equation*}
 \norm{\nabla\Delta^{-1}(\rho_{1,n}-\rho_{2,n})}_{L^2}^2
 \leq C_dR\int\abs{x-y}^2\dd\pi_n(x,y).
\end{equation*}
The normalization introduces no additional mass factor.  Since
$\rho_{i,n}\to\rho_i$ in $L^1$ and in distributions, the fields on the
left converge distributionally to $E_1-E_2$.  Their uniform $L^2$ bound and
weak lower semicontinuity therefore give  \eqref{loeper-L2}.   Finally,
 $\rho_1\leq R$  implies
\begin{equation*}
 \int\abs{E_1-E_2}^2\rho_1
 \leq R\norm{E_1-E_2}_{L^2}^2,
\end{equation*}
which  proves \eqref{loeper}.
\end{proof}

\begin{proof}[Proof of uniqueness]
Let $f_1,f_2$ be two solutions in the class of Theorem~\ref{thm:main} with
the same initial datum, and let
$\Phi_i=(X_i,V_i)$ be their flows.
 By the definition of a Lagrangian solution, $f_i$ is
transported by $\Phi_i$.  The two flows start from the same point in phase
space, and
\begin{align*}
 |V_i(t,y,w)-w|
 &\leq\int_0^t\norm{E_i(s)}_{L^\infty}\dd s,\\
 |X_i(t,y,w)-y-tw|
 &\leq T\int_0^t\norm{E_i(s)}_{L^\infty}\dd s.
\end{align*}
Thus the differences $X_1-X_2$ and $V_1-V_2$ are bounded independently of
$(y,w)$.  In particular, the following quantity is finite and absolutely
continuous:
 \begin{equation}\label{Pdefinition}
 P(t)=\frac{1}{2}\iint f_0(y,w)
\big(\abs{X_1-X_2}^2+\abs{V_1-V_2}^2\big)\dd y\dd w.
\end{equation}
 Write $D_X=X_1-X_2$ and $D_V=V_1-V_2$.
Differentiating \eqref{Pdefinition} gives, for almost every $t$,
\begin{align*}
 P'(t)
 &=\iint f_0D_X\cdot D_V\dd y\dd w\\
 &\quad+\iint f_0D_V\cdot
 \bigl(E_1(t,X_1)-E_2(t,X_2)\bigr)\dd y\dd w.
\end{align*}
The first integral is bounded in absolute value by $P(t)$.  To
estimate the second one, split the force difference as

\begin{equation*}
 E_1(X_1)-E_2(X_2)
 =\big(E_1(X_1)-E_2(X_1)\big)
  +\big(E_2(X_1)-E_2(X_2)\big).
\end{equation*}
Since the solutions are transported by their flows,
\begin{equation*}
 (X_i(t))_\#(f_0\dd y\dd w)=\rho_i(t,x)\dd x.
\end{equation*}
Consequently, the measure
\begin{equation*}
 \pi_t=(X_1(t),X_2(t))_\#(f_0\dd y\dd w)
\end{equation*}
is a coupling of $\rho_1(t)$ and $\rho_2(t)$.  Its quadratic cost
is finite by the preceding bounds and satisfies
\begin{equation*}
 \int_{\R^d\times\R^d}|x-y|^2\dd\pi_t(x,y)
 =\iint f_0|D_X|^2\dd y\dd w\leq2P(t).
\end{equation*}
Set $R(t)=\max_i\norm{\rho_i(t)}_{L^\infty}$.  Applying
\eqref{loeper} gives
\begin{align*}
 \iint f_0|E_1(X_1)-E_2(X_1)|^2\dd y\dd w
 &=\int_{\R^d}|E_1-E_2|^2\rho_1\dd x\\
 &\leq C_dR(t)^2
 \int_{\R^d\times\R^d}|x-y|^2\dd\pi_t(x,y)\\
 &\leq C_dR(t)^2P(t).
\end{align*}
On the other hand, the Lipschitz bound for $E_2$ yields
\begin{align*}
 \iint f_0|E_2(X_1)-E_2(X_2)|^2\dd y\dd w
 &\leq\norm{\nabla E_2(t)}_{L^\infty}^2
 \iint f_0|D_X|^2\dd y\dd w\\
 &\leq2\norm{\nabla E_2(t)}_{L^\infty}^2P(t).
\end{align*}
Therefore
\begin{align*}
 &\iint f_0|E_1(X_1)-E_2(X_2)|^2\dd y\dd w\\
 &\qquad\leq
 C_d\bigl(R(t)^2+\norm{\nabla E_2(t)}_{L^\infty}^2\bigr)P(t).
\end{align*} 
 Using this estimate and Cauchy--Schwarz in the
differentiated expression for $P$ yields
\begin{equation}\label{Pgronwall}
 P'(t)\leq C_d\Big(1+\max_{i=1,2}\norm{\rho_i(t)}_{L^\infty}
                    +\norm{\nabla E_2(t)}_{L^\infty}\Big)P(t)
\end{equation}
 for almost every $t$.

The coefficient belongs to $L^1(0,T)$.  Indeed,
 $\norm{\rho_i(t)}_{L^\infty}\lesssim t^{-d/p}$  with $d/p<1$, and
$\nabla E_i\in L^1_tL^\infty_x$ by \eqref{Eclass}.  Since $P(0)=0$,
Gronwall's inequality gives $P\equiv0$.  Thus the two flows agree
$f_0$-almost everywhere, and consequently $f_1=f_2$.
\end{proof}

\section{Time continuity and continuous dependence}\label{sec:continuity}

We finish the well-posedness proof by showing that the Lagrangian solution is
continuous in time and depends continuously on the initial datum.

\begin{proof}[Proof of time continuity]
The backward maps $(X_{0,t},V_{0,t})$ converge locally uniformly to
$(X_{0,s},V_{0,s})$ as $t\to s$, because
$E\in L^1_tC^1_x$.  If $g\in C_c^\infty(\R^{2d})$, dominated convergence
and volume preservation give
\begin{equation*}
 \norm{g(X_{0,t},V_{0,t})-g(X_{0,s},V_{0,s})}_{L^1\cap L^p}\to0.
\end{equation*}
Approximate $f_0$ in $L^1\cap L^p$ by $g$ and again use volume preservation.
This proves \eqref{fclass}.
\end{proof}

The same argument yields continuous dependence.  More precisely, suppose
$f_{0,n}\to f_0$ in $L^1\cap L^p$ and, for the corresponding envelopes,
 \begin{equation*}
 \sup_n\bigl(\norm{f_{0,n}}_{L^1}+\norm{G_{0,n}}_{L^p}
                    +\norm{G_{\theta,n}}_{L^p}\bigr)<\infty.
\end{equation*} 
The preceding construction gives a common existence time  $T$.  The estimates
of Lemma~\ref{lem:mixing}, the continuity equation, and
$f_{0,n}\to f_0$ in $L^1$ give uniform spatial tightness of the corresponding
densities $\rho_n$.  Lemma~\ref{lem:compactness}, with the same diagonal
argument when the initial traces vary, therefore shows  that every subsequence
contains a further subsequence  such that
\begin{equation}\label{continuous-dependence-density}
 \rho_n\longrightarrow\bar\rho
 \quad\hbox{in }L^r(0,T;C^\beta_{\rm loc}(\R^d)).
\end{equation}
The initial trace of the limit is $\rho_0=\int f_0\,\dd v$: indeed,
$\rho_{0,n}\to\rho_0$ in $L^1_x$ and the uniform time-derivative bound passes
the trace in distributions.

The local/far-field decomposition used in
\eqref{field-convergence} now gives
\begin{equation*}
 E_n=\sigma\nabla\Delta^{-1}\rho_n
 \longrightarrow
 \bar E=\sigma\nabla\Delta^{-1}\bar\rho
 \quad\hbox{in }L^1(0,T;C^{1,\beta'}(O))
\end{equation*}
for every $O\Subset\R^d$ and $0<\beta'<\beta$.  Hence the corresponding
backward flows converge locally uniformly, uniformly in $t\in[0,T]$.
Let $\bar\Phi_{0,t}$ denote the flow generated by $\bar E$.  For $a=1,p$,
volume preservation and approximation of $f_0$ by compactly supported smooth
functions give
\begin{align*}
 &\sup_{t\leq T}
 \norm{f_{0,n}(\Phi^n_{0,t})-f_0(\bar\Phi_{0,t})}_{L^a}\\
 &\leq \norm{f_{0,n}-f_0}_{L^a}
   +\sup_{t\leq T}\norm{
      f_0(\Phi^n_{0,t})-f_0(\bar\Phi_{0,t})}_{L^a}
      \longrightarrow0.
\end{align*}
It follows both that
$\bar\rho=\int f_0(\bar\Phi_{0,t})\,\dd v$ and that
$\bar E=\sigma\nabla\Delta^{-1}\bar\rho$.  Thus the subsequential limit is  a
solution with  initial  datum $f_0$.  Uniqueness identifies  it with $(f,E)$.
Since every subsequence has the same limit,  the whole sequence converges , and
the preceding estimate becomes
\begin{equation*}
 \sup_{t\leq T}\norm{f_n(t)-f(t)}_{L^1\cap L^p}\longrightarrow0.
\end{equation*} 
 Therefore  the solution map is continuous into
$C([0,T];L^1\cap L^p)$ on subsets with a uniform data bound.

\section{Discussion and open problems}\label{sec:endpoints}

We conclude by comparing the result with the bounded-density theory and by
isolating the endpoint and nonuniqueness questions left open by the argument.

\subsection{Beyond the classical bounded-density regime}

For the two-dimensional Euler equation, Yudovich's theorem uses bounded
vorticity to obtain a log-Lipschitz velocity field and then an Osgood
uniqueness argument.  Beyond this class, Bru\`e and
Colombo \cite{BrueColombo} proved nonuniqueness for weak velocity solutions
whose vorticity belongs to the Lorentz space $L^{1,\infty}$.  More recently,
Bru\`e, Colombo, and Kumar \cite{BrueColomboKumar} constructed nonunique weak
solutions with uniformly bounded kinetic energy and vorticity in
$L^\infty_tL^q_x$ for some $q>1$.  These results concern weak solutions
obtained by convex integration and therefore do not contradict Yudovich's
uniqueness theorem.   The survey \cite{Brue} emphasizes both the robustness
of this mechanism and the difficulty of determining what remains true for
finite-$L^q$ vorticity.  For Vlasov--Poisson, the corresponding classical
reference point is Loeper's bounded-density theorem \cite{Loeper}: a uniform
$L^\infty_x$ bound on the macroscopic density yields a quantitative
Wasserstein stability estimate.  Uniqueness criteria allowing certain
unbounded densities were subsequently obtained in \cite{Uniqueness,Yudovich}.
These criteria do not establish uniqueness from the sole
assumption that the density belongs to one fixed $L^p_x$ space with
$p<\infty$; instead, they require control of $L^q_x$ norms for arbitrarily
large finite $q$, with prescribed growth as $q\to\infty$.

Theorem~\ref{thm:main} lies beyond the classical \emph{uniformly}
bounded-density regime.  Indeed, the assumptions allow
$\rho_0\in L^p_x\setminus L^\infty_x$, and the dispersive estimate gives only
 \begin{equation}\label{time-singular-density}
 \norm{\rho(t)}_{L^\infty}\lesssim t^{-d/p}\norm{G_0}_{L^p}.
\end{equation} 
Thus $\rho$ need not belong to $L^\infty(0,T;L^\infty_x)$.  The condition
$p>d$ nevertheless makes the singularity in
\eqref{time-singular-density} integrable, while fractional velocity
regularity yields $E\in L^1(0,T;C^{1,\alpha}_x)$.  Loeper's estimate can then
be applied at each positive time and integrated with a time-dependent
coefficient, as in Section~\ref{sec:uniqueness}.  In this precise sense, the
result passes beyond the classical bounded-density uniqueness class.  We do
not mean that it contains every known Yudovich- or Orlicz-type extension;
the mechanisms and assumptions of those theories are different.

\subsection{The endpoints \texorpdfstring{$p=d$ and $\theta=0$}{p equal to d and theta equal to zero}}

The decisive singularity in the present construction is \eqref{rhoholder}:
\begin{equation*}
 [\rho(t)]_{C^\alpha}\lesssim t^{-d/p-\alpha}.
\end{equation*}
It is time integrable exactly when $d/p+\alpha<1$.  We also need
$\alpha<\theta$ in the fractional velocity estimate.  Hence an admissible
positive $\alpha$ exists under the two strict conditions
\begin{equation*}
 p>d,\qquad \theta>0.
\end{equation*}
For every $p>d$ and every arbitrarily small $\theta>0$, one may choose
\begin{equation*}
 0<\alpha<\min\left\{\theta,1-\frac dp\right\}.
\end{equation*}

At $p=d$, even the $L^\infty$ estimate behaves like $t^{-1}$, so the
time-integrability mechanism used here becomes critical.  At $\theta=0$, the
assumptions provide no positive H\"older modulus for the density, even for
free transport.  Nevertheless, when $p>d$, estimate
\eqref{time-singular-density} formally gives
$\rho\in L^1(0,T;L^\infty_x)$ and hence a time-integrable log-Lipschitz
electric field.  This points toward an Osgood or regular-Lagrangian-flow
framework, but the compact construction and stability of the nonlinear
velocity-to-position projection are presently missing.

\begin{question}[The endpoint $\theta=0$]
\label{op:endpoints}
For $p>d$, does Theorem~\ref{thm:main} remain valid under
the sole weighted envelope condition $G_0\in L^p_x$, with $\theta=0$ and no
positive velocity modulus?  A proof would have to propagate the dispersive
$L^1_tL^\infty_x$ density estimate through a log-Lipschitz or regular
Lagrangian flow without using pointwise $C^1$ Jacobian control for the
characteristic reparametrization.
\end{question}
The restrictions $p>d$ and $\theta>0$ should therefore be understood as
sharp for the present H\"older-flow argument, not as established thresholds
for well-posedness itself.

\subsection{A nonuniqueness problem for Vlasov--Poisson}

The developments for two-dimensional Euler surveyed in \cite{Brue} show
that failure of the classical Yudovich hypothesis can lead, in sufficiently
weak classes, to instability and nonuniqueness.  In particular, the survey
records nonunique weak Euler solutions with vorticity in $C_tL^q_x$ for
$1<q<1+1/6500$, while the general finite-$q$ question remains open.  This
analogy is suggestive for Vlasov--Poisson but should not be pushed too far.
The kinetic equation
has a phase-space transport structure, positivity of $f$, an elliptic
self-consistency constraint, and a second-order characteristic system.  All
of these features impose restrictions absent from the vorticity formulation
of two-dimensional Euler.

To the best of our knowledge, no nonuniqueness construction is known for the
Vlasov--Poisson Cauchy problem in a natural class of nonnegative, finite-mass
weak solutions  on the whole space  with finite energy or comparable velocity
moments.   Nonuniqueness is known in a different, substantially more singular
setting: Binshati and Tudorascu \cite{BinshatiTudorascu} treat a generalized
one-dimensional periodic formulation for probability-measure solutions,
with the force defined through a barycentric projection.  That result does
not address the standard whole-space finite-density class considered here.
 This leads to the following open problem.

\medskip
\noindent\textbf{Open problem (nonuniqueness below the density-stability
threshold).}
Does there exist a nonnegative finite-mass datum $f_0$ for which the
Vlasov--Poisson system admits two distinct distributional solutions with the
same initial datum, satisfying the natural energy and moment bounds, but
whose macroscopic densities lie outside the bounded, Yudovich, and related
stability classes of \cite{Loeper,Uniqueness,Yudovich}?  Can such
nonuniqueness occur among renormalized or regular-Lagrangian solutions, or is
it confined to a still weaker distributional class?

Any positive answer would have to evade the Wasserstein stability mechanism
used in Section~\ref{sec:uniqueness}.  Conversely, a negative answer would
point to a genuinely kinetic rigidity principle extending uniqueness far
beyond all presently known density criteria.  Determining which alternative
holds is, in our view, a fundamental open direction.

\end{document}